\documentclass[a4,amstex,11pt]{article}

\usepackage{geometry}
\usepackage{amsfonts}
\usepackage{graphicx}
\usepackage{epstopdf}
\usepackage{enumitem}
\usepackage{algorithm}
\usepackage{algorithmic}
\usepackage{amsmath,amsfonts,amssymb,mathrsfs,bm}
\usepackage{amsthm}
\usepackage{subcaption}
\ifpdf
  \DeclareGraphicsExtensions{.eps,.pdf,.png,.jpg}
\else
  \DeclareGraphicsExtensions{.eps}
\fi

\makeatletter
\let\@fnsymbol\@arabic
\makeatother

\usepackage{amsopn}
\usepackage{mathtools}
\usepackage{cleveref}
\Crefname{theorem}{Theorem}{Theorems}
\Crefname{proposition}{Proposition}{Propositions}
\Crefname{lemma}{Lemma}{Lemmas}
\Crefname{corollary}{Corollary}{Corollaries}
\Crefname{remark}{Remark}{Remarks}
\Crefname{example}{Example}{Examples}
\Crefname{algorithm}{Algorithm}{Algorithms}
\Crefname{figure}{Figure}{Figures}
\Crefname{equation}{}{}

\theoremstyle{definition}
\newtheorem{definition}{Definition}[section]
\newtheorem{theorem}[definition]{Theorem}
\newtheorem{proposition}[definition]{Proposition}
\newtheorem{lemma}[definition]{Lemma}

\newtheorem{remark}[definition]{Remark}
\newtheorem{example}[definition]{Example}

\title{Tropical linearization and stability analysis of discrete dynamical systems at the tropical origin
}

\author{
Yuki Nishida\thanks{Kyoto Prefectural University, Kyoto, Japan.
(Email: y-nishida@kpu.ac.jp)
}
\and 
Sennosuke Watanabe\thanks{The University of Fukuchiyama, Fukuchiyama, Japan.}
\and 
Yoshihide Watanabe\thanks{Doshisha University, Kyotanabe, Japan.}
}

\begin{document}
\maketitle

% =====================================

%\renewcommand{\abstractname}{\large \textbf{Abstract}}
\begin{abstract}
The tropical semiring is a semiring of extended real numbers, where the operations of `max' and `+' replace the usual addition and multiplication, respectively. Difference equations obtained from the ultradiscrete limit of discrete dynamical systems are described in terms of the tropical semiring. We propose a tropical linearization approach for the stability analysis of difference equations, including those describing ulradiscrete dynamical systems. We show that the fixed point at the tropical origin is asymptotically stable if the maximum eigenvalue of the tropical Jacobian matrix is negative. On the other hand, it is unstable if the maximum eigenvalue of the tropical Jacobian matrix is positive. Since $0$ is the tropical multiplicative identity, these results are analogous to those in the usual linearization process.
\end{abstract}

% =====================================

\section{Introduction}

The tropical semiring, or the max-plus algebra, is a semiring of extended real numbers $\mathbb{R}_{\max}:= \mathbb{R} \cup \{-\infty\}$, where the operations of `max' and `+' replace the usual addition and multiplication, respectively.
It has been applied to the analysis of discrete event systems such as manufacturing~\cite{Cohen1985} and railway systems~\cite{Vries1998}.
In these systems, the operation `max' has an advantage in representing the synchronization of events, and the evolution can be expressed as tropical linear difference equations.
The tropical semiring is also used to describe ultradiscrete dynamical systems~\cite{Tokihiro1996}, in particular, ultradiscrete integrable systems such as box-ball systems~\cite{Takahashi1990}. 
The tropical discretization, which is a combination of positive-valued discretization and ultradiscrete limiting process, of differential equations also provides ultradiscrete dynamical systems~\cite{Murata2013}.
The asymptotic behaviour of ultradiscrete dynamical systems has been studied for specific cases; for example, the negative feedback model~\cite{Gibo2015} and Sel'kov model~\cite{Yamazaki2024}. However, to the best of our knowledge, the general case has not been exploited. Ultradiscrete dynamical systems are described using piecewise-linear difference equations; therefore, their stability analysis based on differentiation is not applicable. 

As in the cases of usual linear differential or difference equations, tropical linear difference equations can be solved using the spectral theory of matrices. The tropical spectral theory originated in the 1960s~\cite{Green1962} and has been developed in many research including the `Cyclicity Theorem' for irreducible matrices~\cite{Cohen1985} and its extension to reducible matrices~\cite{Butkovic2009}.
These results show that solutions to tropical linear difference equations eventually become periodic, with growth rates given by the eigenvalues.

In this paper, we propose a tropical linearization approach for the stability analysis of difference equations, including non-differentiable ones describing ultradiscrete dynamical systems. Specifically, we focus on the fixed point $\bm{\varepsilon}:= (-\infty,\dots,-\infty)^\top$, which is the tropical origin; this is because the shift of the fixed point is not allowed due to the lack of tropical subtraction. 
We first introduce the tropical linear approximation of a map $g:\mathbb{R}_{\max}^n \to \mathbb{R}_{\max}$ based on the distance $d(x,y) := |e^x - e^y|$, resulting the tropical derivative and tropical Jacobian matrix at $\bm{\varepsilon}$. 
Then, we show that the fixed point $\bm{\varepsilon}$ is asymptotically stable if the maximum eigenvalue of the tropical Jacobian matrix is negative. On the other hand, it is unstable if the maximum eigenvalue of the tropical Jacobian matrix is positive. Since $0$ is the tropical multiplicative identity, these results are analogous to those in the usual linearization process.

In the context of tropical functional analysis, tropical integral has been defined as the supremum of a function~\cite{Kolokoltsov1997}.
On the other hand, tropical differentiation is less studied because the tropical subtraction cannot be defined. 
Some studies define tropical differentiation algebraically as a linear operation that decreases the degree of a monomial by one~\cite{Falkensteiner2023,Grigoriev2017}.
Compared to them, our definition of tropical derivative is analytic and also valid for functions that are not tropical polynomials.
The tropical approximation problem has been developed in literature, e.g., in~\cite{Krivulin2025}, but note that we do not aim at finding a better approximation under some criterion; we just consider the tropical linear approximation that is analytically derived.

The rest of the paper is organized as follows. In Section 2, we introduce basic definitions and results on linear algebra over the tropical semiring. In Section 3, we present the definition of tropical linear approximation and tropical derivative at $\bm{\varepsilon}$. In Section 4, we derive the main results of this paper: asymptotic stability and instability of tropical-linearly approximated difference equations. We also present some examples to which our results can be applied. In Section 5, we conclude the paper with some remarks on the fixed point other than $\bm{\varepsilon}$.

\section{Tropical semiring}
The tropical semiring, or the max-plus algebra, is a set of numbers $\mathbb{R}_{\max}:= \mathbb{R} \cup \{-\infty\}$, where the tropical addition $\oplus$ and multiplication $\otimes$ are defined as 
\begin{align*}
	a \oplus b:= \max(a,b), \qquad a \otimes b:= a+ b
\end{align*}
for $a,b \in \mathbb{R}_{\max}$.
The identities of addition and multiplication are $\varepsilon:= -\infty$ and $0$, respectively.
In the tropical semiring, the subtraction operation cannot be defined because the inverse element for the addition does not exist, and the tropical division is defined as $a \oslash b:= a - b$.
The tropical power $a^{\otimes r}$ for $r \in \mathbb{N}$ is defined as the $r$-fold tropical product, i.e., 
\begin{align*}
	a^{\otimes r} = \underbrace{a \otimes a \otimes \cdots \otimes a}_{r\,\text{times}} = ra.
\end{align*}
This definition is easily extended to $r \in \mathbb{R}$ as $a^{\otimes r} = ra$.
For details on the tropical semiring, refer to textbooks~\cite{Baccelli1992,Butkovic2010,Heidergott2005,Joswig2021,Maclagan2015}.

Let $\mathbb{R}_{\max}^n$ and $\mathbb{R}_{\max}^{m \times n }$ denote the sets of all $n$-dimensional column vectors and $m$-by-$n$ matrices with entries in $\mathbb{R}_{\max}$, respectively.
For $A, B \in \mathbb{R}_{\max}^{m \times n}$, 
the matrix sum $A \oplus B \in \mathbb{R}_{\max}^{m \times n}$ is defined by
\begin{align*}
[A \oplus B]_{ij} = [A]_{ij} \oplus [B]_{ij},
\end{align*}
where $[A]_{ij}$ denotes the $(i,j)$ entry of the matrix $A$.
For $A \in \mathbb{R}_{\max}^{l \times m}$ and $B \in \mathbb{R}_{\max}^{m \times n}$, 
the matrix product $A \otimes B \in \mathbb{R}_{\max}^{l \times n}$ is defined by
\begin{align*}
[A \otimes B]_{ij} = \bigoplus_{k=1}^m [A]_{ik} \otimes [B]_{kj}.
\end{align*}
For $A \in \mathbb{R}_{\max}^{m \times n}$ and $\alpha \in \mathbb{R}_{\max}$,
the scalar multiplication $\alpha \otimes A \in \mathbb{R}_{\max}^{m \times n}$ is defined by
\begin{align*}
[\alpha \otimes A]_{ij} = \alpha \otimes [A]_{ij}.
\end{align*}
These operations can be applied to column vectors by regarding them as one-column matrices.
The tropical zero vector is $\bm{\varepsilon} := (\varepsilon,\dots,\varepsilon)^{\top} \in \mathbb{R}_{\max}^{n}$.
For vectors $\bm{x}=(x_1,x_2,\dots,x_n)^\top, \bm{y}=(y_1,y_2,\dots,y_n)^\top \in \mathbb{R}_{\max}^n$, the inequality $\bm{x} \leq \bm{y}$ means $x_i \leq y_i$ for all $i=1,2,\dots,n$. Other types of inequalities are understood similarly.

Linearity over the tropical semiring is defined from the viewpoint of semimodules.
A nonempty subset $U \subset \mathbb{R}_{\max}^n$ is called a tropical subspace of $\mathbb{R}_{\max}^n$ if $a \otimes \bm{x} \oplus b \otimes \bm{y} \in U$ for any $\bm{x},\bm{y} \in U$ and $a,b \in \mathbb{R}_{\max}$.
For tropical subspaces $U_1 \subset \mathbb{R}_{\max}^n$ and $U_2 \subset \mathbb{R}_{\max}^m$,
a map $f: U_1 \to U_2$ is said to be tropical linear if 
\begin{align*}
	f(a \otimes \bm{x} \oplus b \otimes \bm{y}) = a \otimes f(\bm{x}) \oplus b \otimes f(\bm{y})
\end{align*}
for any $\bm{x},\bm{y} \in U_1$ and $a,b \in \mathbb{R}_{\max}$.
Obviously, $A \in \mathbb{R}_{\max}^{m \times n}$ defines a tropical linear map $T_A: \mathbb{R}_{\max}^n \to \mathbb{R}_{\max}^m$ by $T_A(\bm{x}) = A \otimes \bm{x}$.

For $A \in \mathbb{R}_{\max}^{n \times n}$, a scalar $\lambda \in \mathbb{R}_{\max}$ is called an eigenvalue of $A$ if there exists $\bm{x} \in \mathbb{R}_{\max}^n \setminus \{\bm{\varepsilon}\}$ satisfying
\begin{align*}
	A \otimes \bm{x} = \lambda \otimes \bm{x}.
\end{align*}
Such a vector $\bm{x}$ is an eigenvector of $A$ with respect to $\lambda$.

To understand the eigenvalues of tropical matrices, we use their associated digraphs.  
For $A \in \mathbb{R}_{\max}^{n \times n}$, the weighted digraph $\mathcal{G}(A)$ is defined by the vertex set $V = \{1,2,\dots,n\}$ and the edge set $E = \{ (i,j) \mid [A]_{ij} \neq \varepsilon\}$ with weight $w((i,j)) = [A]_{ij}$ for $(i,j) \in E$.
A path in $\mathcal{G}(A)$ is a sequence of vertices $(i_0,i_1,\dots,i_{\ell})$ such that $(i_{k-1},i_k) \in E$ for $k=1,2,\dots,\ell$;
 if $i_0 = i_\ell$, this path is called a circuit.
The set of edges in a path $\mathcal{P} = (i_0,i_1,\dots,i_{\ell})$ is $E(\mathcal{P}) := \{(i_{k-1},i_k) \mid k=1,2,\dots,\ell\}$.
The length, weight, and average weight of $\mathcal{P}$ are defined as $\ell(\mathcal{P}) := \ell$, $w(\mathcal{P}) := \sum_{k=1}^{\ell} w((i_{k-1},i_k))$ and $\mathrm{ave}(\mathcal{P}) = w(\mathcal{P})/\ell(\mathcal{P})$, respectively.

\begin{theorem}[\cite{Green1979}]
The maximum eigenvalue of $A \in \mathbb{R}_{\max}^{n \times n}$ is identical to the maximum average weight of all circuits in $\mathcal{G}(A)$.
\end{theorem}

A matrix $A \in \mathbb{R}_{\max}^{n \times n}$ is called irreducible if $\mathcal{G}(A)$ is strongly connected, i.e., for any $i,j \in V$ there exists a path from $i$ to $j$. If $A \in \mathbb{R}_{\max}^{n \times n}$ is irreducible, then $A$ has a unique eigenvalue~\cite{Green1979}. 
In the digraph $\mathcal{G}(A)$, let $V^c(A)$ and $E^c(A)$ be the sets of all vertices and edges, respectively, contained in some maximum average weight circuit. The subgraph $\mathcal{G}^c(A) = (V^c(A), E^c(A))$ is called the critical digraph of $A$. The cyclicity of a strongly connected component of $\mathcal{G}^c(A)$ is the greatest common divisor of the lengths of all circuits contained in that component, and the cyclicity of $A$ is the least common multiple of the cyclicities of all strongly connected components of $\mathcal{G}^c(A)$.

For a square matrix $A \in \mathbb{R}_{\max}^{n \times n}$ and $k \in \mathbb{N}$, the matrix power is defined as
\begin{align*}
	A^{\otimes k} = \underbrace{A \otimes A \otimes \cdots \otimes A}_{k\,\text{times}}.
\end{align*}
By convention, we define $A^{\otimes 0}$ as the tropical identity matrix, that is, the square matrix with zeros on the diagonal and $\varepsilon$ elsewhere.
The following fact, known as the Cyclicity Theorem, states that the sequence of matrix powers becomes periodic.

\begin{theorem}[Cyclicity Theorem, \cite{Cohen1985}]		\label{thm:cyc}
Let $A \in \mathbb{R}_{\max}^{n \times n}$ be an irreducible matrix and $\lambda$ and $\sigma$ be its unique eigenvalue and cyclicity, respectively. Then, there exist an integer $T \geq 1$ such that
\begin{align*}
	A^{\otimes (t+\sigma)} = \lambda^{\otimes \sigma} \otimes A^{\otimes t}
\end{align*}
for any $t \geq T$. 
\end{theorem}

For $A \in \mathbb{R}_{\max}^{n \times n}$ and $k \in \mathbb{N}$, the $(i,j)$ entry of $A^{\otimes k}$ is identical to the maximum weight of all paths from vertex $i$ to $j$ with length $k$ in $\mathcal{G}(A)$.
If the maximum (average) weight of the circuits in $\mathcal{G}(A)$ is nonpositive, the maximum weight of all paths from $i$ to $j$ is attained by the one with length up to $n$.

\section{Tropical linear approximation}
Let us consider a map $g: \mathbb{R}_{\max}^n \to \mathbb{R}_{\max}$ satisfying $g(\bm{\varepsilon}) = \varepsilon$.
In this section, we describe a tropical linear approximation $g(\bm{x}) \approx \bm{a} \otimes \bm{x}$ at $\bm{\varepsilon}$, where $\bm{a} \in \mathbb{R}_{\max}^{1 \times n}$.

For $\bm{x} = (x_1,x_2,\dots,x_n)^\top, \bm{y} = (y_1,y_2,\dots,y_n)^\top \in \mathbb{R}_{\max}^n$, the distance between them is defined as
\begin{align}
	d(\bm{x}, \bm{y}) := \max_{i =1,2,\dots,n} |e^{x_i} -e^{y_i}|.		\label{eq:dist}
\end{align}
This distance is designed to handle $\varepsilon$ by setting $e^{\varepsilon} = 0$.
The magnitude of $\bm{x}$ is defined as
\begin{align}
	\|\bm{x}\| := d(\bm{x}, \bm{\varepsilon}) = \max_{i =1,2,\dots,n} e^{x_i} = e^{\max_i x_i}.		\label{eq:norm}
\end{align}
Note that $\|\bm{x}\|$ is not a usual norm because
\begin{align}
	\|a \otimes \bm{x}\| = \max_{i =1,2,\dots,n} |e^{a+x_i}| = e^a \|\bm{x}\|	\label{eq:normsp}
\end{align}
for $a \in \mathbb{R}_{\max}$.
We also note that the function $\|\cdot\|$ is monotonic, that is, if $\bm{x} \leq \bm{y}$, then $\|\bm{x}\| \leq \|\bm{y}\|$.

A tropical linear form $\bm{a} \otimes \bm{x}$ given by $\bm{a} \in \mathbb{R}_{\max}^{1 \times n}$ is called a tropical linear approximation of $g(\bm{x})$ at $\bm{\varepsilon}$, expressed as $g(\bm{x}) \approx \bm{a} \otimes \bm{x}$, if 
$d(g(\bm{x}), \bm{a} \otimes \bm{x}) = o(\|\bm{x}\|)$, that is, 
\begin{align*}
	\lim_{\bm{x} \to \bm{\varepsilon}} \frac{d(g(\bm{x}), \bm{a} \otimes \bm{x})}{\|\bm{x}\|} = 0.
\end{align*}
For $\bm{x} = (x_1,x_2,\dots,x_n)^\top \in \mathbb{R}_{\max}^n$, let $\bm{x}(i)$ be the vector whose $i$th entry is $x_i$ and other entries are $\varepsilon$.

\begin{proposition}	\label{thm:linear}
Suppose that $g: \mathbb{R}_{\max}^n \to \mathbb{R}_{\max}$ satisfies $g(\bm{\varepsilon}) = \varepsilon$. If $\bm{a} \otimes \bm{x}$ is a tropical linear approximation of $g(\bm{x})$ at $\bm{\varepsilon}$ for some $\bm{a} = (a_1,a_2,\dots,a_n) \in \mathbb{R}_{\max}^{1 \times n}$, then 
\begin{align*}
	a_i = \lim_{x_i \to \varepsilon} g(\bm{x}(i)) \oslash x_i
\end{align*}
for $i = 1,2,\dots,n$.
\end{proposition}

\proof
Let us assume that $\bm{a} \otimes \bm{x}$ is a tropical linear approximation of $g(\bm{x})$. 
When $x_j = \varepsilon$ for $j \neq i$, we have $\bm{a} \otimes \bm{x} = \max_{j} (a_j + x_j) = a_i+x_i$ and $\max_{j} x_j = x_i.$
Then, by \eqref{eq:dist} and \eqref{eq:norm}, we have
\begin{align*}
	\frac{d(g(\bm{x}), \bm{a} \otimes \bm{x})}{\|\bm{x}\|} = \frac{|e^{g(\bm{x}(i))} - e^{a_i + x_i}|}{e^{x_i}}.
\end{align*}
Hence, we obtain
\begin{align*}
	\lim_{x_i \to \varepsilon} \frac{|e^{g(\bm{x}(i))} - e^{a_i + x_i}|}{e^{x_i}} = 0,
\end{align*}
which implies
\begin{align*}
	\lim_{x_i \to \varepsilon} |e^{g(\bm{x}(i)) - x_i} - e^{a_i}| = 0.
\end{align*}
Recalling that $g(\bm{x}(i)) - x_i = g(\bm{x}(i)) \oslash x_i$, we have
\begin{align*}
	a_i = \lim_{x_i \to \varepsilon} g(\bm{x}(i)) \oslash x_i.
\end{align*}
As the above argument is valid for $i = 1,2,\dots,n$, we have completed the proof.
\endproof

Based on Proposition~\ref{thm:linear}, 
when $g: \mathbb{R}_{\max}^n \to \mathbb{R}_{\max}$ satisfies $g(\bm{\varepsilon}) = \varepsilon$, we define the tropical derivative of $g$ at $\bm{\varepsilon}$ as
\begin{align*}
	D_{i,\bm{\varepsilon}} g := \lim_{x_i \to \varepsilon} g(\bm{x}(i)) \oslash x_i
\end{align*}
for $i=1,2,\dots,n$.
One may recall that the conventional partial derivative of $\tilde{g}: \mathbb{R}^n \to \mathbb{R}$ at $\bm{p}=(p_1,p_2,\dots,p_n)^\top \in \mathbb{R}^n$ is given by
\begin{align}
	\frac{\partial g}{\partial x_i}(\bm{p}) = \lim_{x_i \to p_i} \frac{\tilde{g}((p_1,\dots,x_i,\dots,p_n)^\top) - \tilde{g}(\bm{p})}{x_i-p_i}	\label{eq:usualpd}
\end{align}
for $i=1,2,\dots,n$. If $\bm{p}=(0,0,\dots,0)^\top$ and $\tilde{g}(\bm{p}) = 0$, then the right-hand side of~\eqref{eq:usualpd} becomes 
\begin{align*}
\lim_{x_i \to 0} \frac{\tilde{g}((0,\dots,x_i,\dots,0)^\top)}{x_i},
\end{align*}
from which our tropical derivative comes.

\begin{example}
Let us consider a tropical polynomial function $g(x) = \bigoplus_{k=1}^m c_k \otimes x^{\otimes k}$. Then, 
\begin{align}
	\lim_{x \to \varepsilon} g(x) \oslash x
	= c_1 \oplus \left( \bigoplus_{k=2}^m c_k \otimes x^{\otimes (k-2)} \right) \otimes x.	\label{eq:poly}
\end{align}
The second term on the right-hand side of~\eqref{eq:poly} converges to $\varepsilon$ as $x \to \varepsilon$. Hence, the tropical linear approximation of $g$ is $g(\bm{x}) \approx c_1 \otimes x$, similar to the usual algebra.
\end{example}

\begin{remark}
The converse of Proposition~\ref{thm:linear} is not true. Indeed, let us consider 
\begin{align*}
	g(x,y) = x \oplus y \oplus (x \otimes y)^{\otimes \frac{1}{3}}.
\end{align*}
It is easily verified that
\begin{align*}
	\lim_{x \to \varepsilon} g(x,\varepsilon) \oslash x = \lim_{y \to \varepsilon} g(\varepsilon,y) \oslash y = 0.
\end{align*}
However, when $x=y$, by setting $\bm{x} = (x,x)^\top$, we have 
\begin{align*}
	\lim_{x \to \varepsilon}\frac{|e^{g(x,x)} - e^{0\otimes x\oplus 0\otimes x}|}{\|\bm{x}\|}
	= \lim_{x \to \varepsilon}\frac{|e^{x \oplus x^{\otimes \frac{2}{3}}} - e^x|}{e^x}
	= \lim_{x \to \varepsilon} |e^{-\frac{1}{3}x} -1|
	= +\infty,
\end{align*}
which means that $g(x,y)$ cannot be tropical-linearly approximated.
\end{remark}

Based on the tropical linear approximation, the following lemma sets an upper bound of $g(\bm{x})$.

\begin{lemma}		\label{lem:ub}
Let $g: \mathbb{R}_{\max}^n \to \mathbb{R}_{\max}$, $\bm{a} \in \mathbb{R}_{\max}^{1 \times n}$ and $\alpha > 0$. If $\bm{x} \in \mathbb{R}_{\max}^n$ satisfies $d(g(\bm{x}), \bm{a} \otimes \bm{x}) \leq \alpha \|\bm{x}\|$, then 
\begin{align}
	g(\bm{x}) \leq \log(1+\sqrt{\alpha}) \otimes \bm{a} \otimes \bm{x} \oplus \log(\alpha+\sqrt{\alpha}) \otimes \bm{0} \otimes \bm{x},	\label{eq:gbound}
\end{align}
where $\bm{0} = (0,\dots,0) \in \mathbb{R}_{\max}^{1 \times n}$.
\end{lemma}

\proof
Let us assume that $\bm{x} \in \mathbb{R}_{\max}^n$ satisfies
\begin{align*}
	d(g(\bm{x}), \bm{a} \otimes \bm{x}) \leq \alpha \|\bm{x}\|.
\end{align*}
By noting that $d(g(\bm{x}), \bm{a} \otimes \bm{x}) = |e^{g(\bm{x})} - e^{\bm{a} \otimes \bm{x}}|$ and 
$\|\bm{x}\| = e^{\max_i x_i} = e^{\bm{0} \otimes \bm{x}}$, we have
\begin{align*}
	-\alpha e^{\bm{0} \otimes \bm{x}} \leq e^{g(\bm{x})} - e^{\bm{a} \otimes \bm{x}} \leq \alpha e^{\bm{0} \otimes \bm{x}},
\end{align*}
which leads to
\begin{align*}
	e^{g(\bm{x})} \leq e^{\bm{a} \otimes \bm{x}} + \alpha e^{\bm{0} \otimes \bm{x}}.
\end{align*}
\begin{enumerate}
\item[(i)] If $\bm{a} \otimes \bm{x} \geq \log\sqrt{\alpha} \otimes \bm{0} \otimes \bm{x}$, 
that is, $(-\log\sqrt{\alpha} ) \otimes \bm{a} \otimes \bm{x} \geq \bm{0} \otimes \bm{x}$,
then
\begin{align*}
	e^{\bm{0} \otimes \bm{x}} \leq e^{(-\log\sqrt{\alpha} ) \otimes \bm{a} \otimes \bm{x}}
	= e^{(-\log\sqrt{\alpha} ) + (\bm{a} \otimes \bm{x})}
	= e^{\log \sqrt{\alpha}^{-1}} e^{\bm{a} \otimes \bm{x}}
	= (\sqrt{\alpha})^{-1} e^{\bm{a} \otimes \bm{x}}.
\end{align*}
Hence, we obtain
\begin{align*}
	e^{\bm{a} \otimes \bm{x}} + \alpha e^{\bm{0} \otimes \bm{x}}
	\leq (1+ \sqrt{\alpha})e^{\bm{a} \otimes \bm{x}}
	= e^{\log (1+ \sqrt{\alpha})} e^{\bm{a} \otimes \bm{x}}
	= e^{\log(1+\sqrt{\alpha}) \otimes \bm{a} \otimes \bm{x}},
\end{align*}
leading to
\begin{align*}
	g(\bm{x}) \leq \log(1+\sqrt{\alpha}) \otimes \bm{a} \otimes \bm{x}.
\end{align*}

\item[(ii)] If $\bm{a} \otimes \bm{x} \leq \log\sqrt{\alpha} \otimes \bm{0} \otimes \bm{x}$, then
\begin{align*}
	e^{\bm{a} \otimes \bm{x}} \leq e^{\log\sqrt{\alpha} \otimes \bm{0} \otimes \bm{x}}
	= e^{(\log\sqrt{\alpha}) + (\bm{0} \otimes \bm{x})}
	= e^{\log \sqrt{\alpha}} e^{\bm{0} \otimes \bm{x}}
	= \sqrt{\alpha}\, e^{\bm{0} \otimes \bm{x}}.
\end{align*}
Hence, we obtain
\begin{align*}
	e^{\bm{a} \otimes \bm{x}} + \alpha e^{\bm{0} \otimes \bm{x}}
	\leq (\sqrt{\alpha} + \alpha)e^{\bm{0} \otimes \bm{x}}
	= e^{\log (\sqrt{\alpha} + \alpha)} e^{\bm{0} \otimes \bm{x}}
	= e^{\log(\alpha+\sqrt{\alpha}) \otimes \bm{0} \otimes \bm{x}},
\end{align*}
leading to
\begin{align*}
	g(\bm{x}) \leq \log(\alpha+\sqrt{\alpha}) \otimes \bm{0} \otimes \bm{x}.
\end{align*}
\end{enumerate}
Since one of the two cases must occur, \eqref{eq:gbound} is proved.
\endproof

\section{Stability analysis}

Let us consider a map $\bm{f}: \mathbb{R}_{\max}^n \to \mathbb{R}_{\max}^n$ and the dynamical system defined by the following difference equation:
\begin{align}
	\bm{x}^{(t+1)} = \bm{f}(\bm{x}^{(t)}), \qquad t \in \mathbb{Z}_{\geq 0},	\label{eq:dyn}
\end{align}
where $\mathbb{Z}_{\geq 0}$ is the set of nonnegative integers.
We denote the $i$th entry of $\bm{x}^{(t)}$ as $x^{(t)}_i$.
Assume that $\bm{\varepsilon}$ is a fixed point of~\eqref{eq:dyn}, i.e., $\bm{f}(\bm{\varepsilon}) = \bm{\varepsilon}$.
The fixed point $\bm{\varepsilon}$ is said to be stable if for all $\alpha > 0$ there exists $\delta > 0$ such that 
\begin{align*}
	\|\bm{x}^{(0)}\| < \delta \quad \Rightarrow \quad \|\bm{x}^{(t)}\| < \alpha \text{ for all } t \in \mathbb{Z}_{\geq 0};
\end{align*}
otherwise, it is unstable.
Furthermore, $\bm{\varepsilon}$ is said to be asymptotically stable if it is stable and there exists $\delta' > 0$ such that
\begin{align*}
	\|\bm{x}^{(0)}\| < \delta' \quad \Rightarrow \quad \lim_{t \to \infty} \bm{x}^{(t)} = \bm{\varepsilon}.
\end{align*}

The map $\bm{f}$ is represented as a tuple of $n$ functions $(f_1,f_2,\dots,f_n)$ defined by $\bm{f}(\bm{x}) = (f_1(\bm{x}),f_2(\bm{x}),\dots,f_n(\bm{x}))^\top$ for $\bm{x} \in \mathbb{R}_{\max}^n$. The tropical Jacobian matrix $J_{\bm{\varepsilon}}\bm{f} \in \mathbb{R}_{\max}^{n \times n}$ at $\bm{\varepsilon}$ is defined by
\begin{align*}
	[J_{\bm{\varepsilon}}\bm{f}]_{ij} = D_{j,\bm{\varepsilon}} f_i.
\end{align*}
If $f_i$ is tropical-linearly approximated for $i=1,\dots,n$, then $\bm{f}$ is said to be tropical-linearly approximated at $\bm{\varepsilon}$, expressed as 
\begin{align*}
	\bm{f}(\bm{x}) \approx J_{\bm{\varepsilon}} \bm{f} \otimes \bm{x}.
\end{align*}
The main result of this study is presented below.

\begin{theorem}	\label{thm:main1}
Suppose that $\bm{\varepsilon}$ is a fixed point of~\eqref{eq:dyn} and $\bm{f}$ is tropical-linearly approximated at $\bm{\varepsilon}$.
If the maximum eigenvalue of $J_{\bm{\varepsilon}} \bm{f}$ is negative, then $\bm{\varepsilon}$ is asymptotically stable.
\end{theorem}

\proof
Let $M_0$, $\lambda$, and $\sigma$ be the maximum entry, maximum eigenvalue, and cyclicity of $J_{\bm{\varepsilon}} \bm{f}$, respectively.
Then, by noting that $\lambda < 0$ from our assumption,
we can choose $\alpha > 0$ to be sufficiently small such that $\log(1+\sqrt{\alpha}) < -\lambda$ and $\log(\alpha+\sqrt{\alpha}) < n\lambda-M_1$, where $M_1= \max(n(M_0-\lambda),0)$.
We define
\begin{align*}
	A = \log(1+\sqrt{\alpha}) \otimes J_{\bm{\varepsilon}} \bm{f} \oplus \log(\alpha+\sqrt{\alpha}) \otimes O,	
\end{align*}
where $O$ is the $n$-by-$n$ matrix with all its entries being $0$.
The maximum entry of $A$ is at most $\max(M_0-\lambda,0)$.
Additionally, the maximum eigenvalue of $A$ is negative.
Indeed, let us consider any circuit $\mathcal{C}$ in $\mathcal{G}(A)$.
If none of the edges in $E(\mathcal{C})$ have weight $\log(\alpha+\sqrt{\alpha})$, then all of them should come from finite entries of $\log(1+\sqrt{\alpha}) \otimes J_{\bm{\varepsilon}} \bm{f}$. Hence, the average weight of $\mathcal{C}$ in $\mathcal{G}(A)$ is identical to that in $\mathcal{G}(J_{\bm{\varepsilon}} \bm{f})$ augmented by $\log(1+\sqrt{\alpha})$. This leads to
\begin{align*}
	\mathrm{ave}(\mathcal{C}) \leq \lambda + \log(1+\sqrt{\alpha}) < 0,
\end{align*}
where the equality of the first inequality holds if $\mathcal{C}$ is the maximum average weight circuit in $\mathcal{G}(J_{\bm{\varepsilon}} \bm{f})$.
On the other hand, if the weight of some edge $e \in E(\mathcal{C})$ is $\log(\alpha+\sqrt{\alpha})$, then 
\begin{align*}
	w(\mathcal{C}) = w(e) + \sum_{e' \in E(\mathcal{C}) \setminus \{e\}} w(e') 
	&\leq \log(\alpha+\sqrt{\alpha}) + (n-1) \cdot \max(M_0 - \lambda,0) \\
	&= \log(\alpha+\sqrt{\alpha}) + M_1 - \max(M_0 - \lambda,0) \\
	&< n\lambda \\
	&\leq \ell(\mathcal{C}) \lambda.
\end{align*}
Hence, the average weight of $\mathcal{C}$ is smaller than $\lambda$.
Thus, the maximum eigenvalue of $A$ is $\lambda_1:= \lambda + \log(1+\sqrt{\alpha})$.
Since $\lambda_1 < 0$ from our choice of $\alpha$, the (average) weights of all circuits in $\mathcal{G}(A)$ are negative.
Moreover, by recalling the definition of the cyclicity, the cyclicity of $A$ is $\sigma$ because the set of maximum average weight circuits in $\mathcal{G}(A)$ coincides with that in $\mathcal{G}(J_{\bm{\varepsilon}}\bm{f})$.

Since $\bm{f}$ is tropical-linearly approximated at $\bm{\varepsilon}$, we have 
\begin{align*}
	\lim_{\bm{x} \to \bm{\varepsilon}} \frac{d(f_i(\bm{x}), [J_{\bm{\varepsilon}}\bm{f} \otimes \bm{x}]_i)}{\|\bm{x}\|} = 0
\end{align*}
for all $i=1,2,\dots,n$.
This means that for all $i=1,2,\dots,n$ there exists $\delta_i > 0$ such that
\begin{align*}
	\|\bm{x}\| < \delta_i \quad \Rightarrow \quad d(f_i(\bm{x}), [J_{\bm{\varepsilon}}\bm{f} \otimes \bm{x}]_i) \leq \alpha \|\bm{x}\|.
\end{align*}
In particular, by setting $\delta_0 = \min(\alpha, \delta_1,\delta_2,\dots,\delta_n)$, we have
\begin{align}
	\|\bm{x}\| < \delta_0 \quad \Rightarrow \quad \max_{i=1,2,\dots,n} d(f_i(\bm{x}), [J_{\bm{\varepsilon}}\bm{f} \otimes \bm{x}]_i) \leq \alpha \|\bm{x}\|.
	\label{eq:linapprox0}
\end{align}
Applying Lemma~\ref{lem:ub} to each entry of $\bm{f}(\bm{x})$, we obtain
\begin{align}
	\max_{i=1,2,\dots,n} d(f_i(\bm{x}), [J_{\bm{\varepsilon}}\bm{f} \otimes \bm{x}]_i) \leq \alpha \|\bm{x}\|
	 \quad \Rightarrow \quad 
	 \bm{f}(\bm{x}) \leq A \otimes \bm{x}.			\label{eq:lineval}
\end{align}

Let $\delta = \delta_0e^{-M_1}$ and consider the initial value problem of~\eqref{eq:dyn} starting with any $\bm{x}^{(0)} \in \mathbb{R}_{\max}^n$ such that $\|\bm{x}^{(0)}\| < \delta$.
By induction on $t$, we show that
\begin{align}
	\bm{x}^{(t)} \leq A^{\otimes t} \otimes \bm{x}^{(0)}	\label{eq:ind}
\end{align}
for any $t \in \mathbb{Z}_{\geq 0}$.
The case where $t=0$ is trivial because $A^{\otimes 0}$ is the tropical identity matrix.
Let us assume that \eqref{eq:ind} is satisfied for some $t \in \mathbb{Z}_{\geq 0}$. 
Since the weights of all circuits in $\mathcal{G}(A)$ are negative, the entries of $A^{\otimes t}$ are at most $M_1 = n \cdot \max(M_0-\lambda,0)$, see the last paragraph of Section 2.
Note that this bound $M_1$ is determined only by $J_{\bm{\varepsilon}}\bm{f}$ and does not depend on the value of $\alpha$ used to construct the matrix $A$.
Because of the assumption of induction, the monotonicity of $\|\cdot \|$, and \eqref{eq:normsp},
we have
\begin{align*}
	\|\bm{x}^{(t)}\| \leq \| A^{\otimes t} \otimes \bm{x}^{(0)}\| \leq \| M_1 \otimes \bm{x}^{(0)}\| 
	= e^{M_1} \|\bm{x}^{(0)}\| 
	<e^{M_1}\delta = 
	 \delta_0.
\end{align*}
Using \eqref{eq:linapprox0}, \eqref{eq:lineval} and \eqref{eq:ind}, we obtain
\begin{align*}
	\bm{x}^{(t+1)} = \bm{f}(\bm{x}^{(t)}) \leq A \otimes \bm{x}^{(t)}
	\leq A\otimes (A^{\otimes t} \otimes \bm{x}^{(0)}) = A^{\otimes (t+1)} \otimes \bm{x}^{(0)},
\end{align*}
which proves \eqref{eq:ind} for $t+1$.

We have shown that $\|\bm{x}^{(t)}\| < \delta_0 \leq \alpha$ for all $t \in \mathbb{Z}_{\geq 0}$, which means $\bm{\varepsilon}$ is stable.
Since all entries of $A$ are finite, $A$ is irreducible. By Theorem~\ref{thm:cyc}, there exist $T \geq 1$ such that
\begin{align*}
	A^{\otimes (k\sigma+T+r)} = \lambda_1^{\otimes k\sigma} \otimes A^{\otimes (T+r)}
\end{align*}
for any $k \in \mathbb{Z}_{\geq 0}$ and $r=0,1,\dots,\sigma-1$.
This leads to
\begin{align*}
	\bm{x}^{(k\sigma+T+r)} \leq \lambda_1^{\otimes k\sigma} \otimes A^{\otimes (T+r)} \otimes \bm{x}^{(0)}
\end{align*}
for any $k \in \mathbb{Z}_{\geq 0}$ and $r=0,1,\dots,\sigma-1$.
Because $\lambda_1 < 0$, we have $\lim_{k \to \infty} \bm{x}^{(k\sigma+T+r)} = \bm{\varepsilon}$ for any $r$, which means that $\lim_{t \to \infty} \bm{x}^{(t)} = \bm{\varepsilon}$.
Thus, $\bm{\varepsilon}$ is asymptotically stable.
\endproof

To derive a condition for the instability of the fixed point $\bm{\varepsilon}$, we introduce the following result in the visualization of tropical matrices.

\begin{lemma}[\cite{Sergeev2009}]	\label{lem:visualize}
Let $\lambda \in \mathbb{R}$ be the maximum eigenvalue of $A \in \mathbb{R}_{\max}^{n \times n}$.
Then, there exists a vector $(x_1,x_2,\dots,x_n)^\top \in \mathbb{R}^n$ such that
\begin{align*}
	\begin{cases} [A]_{ij} \otimes x_j 	 = \lambda \otimes x_i, & \text{if } (i,j) \in E^c(A), \\ 
	[A]_{ij} \otimes x_j < \lambda \otimes x_i, & \text{otherwise}. \end{cases} 
\end{align*}
\end{lemma}

\begin{theorem}	\label{thm:main2}
Suppose that $\bm{\varepsilon}$ is a fixed point of~\eqref{eq:dyn} and $\bm{f}$ is tropical-linearly approximated at $\bm{\varepsilon}$. If the maximum eigenvalue of $J_{\bm{\varepsilon}} \bm{f}$ is positive, then $\bm{\varepsilon}$ is unstable.
\end{theorem}

\proof
Let $\lambda > 0$ be the maximum eigenvalue of $J_{\bm{\varepsilon}} \bm{f}$. For brevity, we write $V^c = V^c(J_{\bm{\varepsilon}} \bm{f})$ and $E^c = E^c(J_{\bm{\varepsilon}} \bm{f})$.
By Lemma~\ref{lem:visualize}, there exists a vector $\hat{\bm{x}} = (\hat{x}_1,\hat{x}_2,\dots,\hat{x}_n)^\top \in \mathbb{R}^n$ such that
\begin{align}
	\begin{cases}
	[J_{\bm{\varepsilon}} \bm{f}]_{ij} \otimes \hat{x}_j  = \lambda \otimes \hat{x}_i, & \text{if } (i,j) \in E^c, \\ 
	[J_{\bm{\varepsilon}} \bm{f}]_{ij} \otimes \hat{x}_j < \lambda \otimes \hat{x}_i, & \text{otherwise}. \end{cases} 	\label{eq:visualize}
\end{align}
Let us define
\begin{align}
	\eta = \min\left(\lambda,\ \min_{(i,j) \not\in E^c} ((\lambda \otimes \hat{x}_i) - ([J_{\bm{\varepsilon}} \bm{f}]_{ij} \otimes \hat{x}_j))\right).	\label{eq:defeta}
\end{align}
We choose $\alpha > 0$ to be sufficiently small such that 
\begin{align}
&\log(1+\sqrt{\alpha}) < \frac{\eta}{2},		\label{eq:alpha1u}\\
&\log (1-\sqrt{\alpha}) > -\frac{\eta}{2},	\label{eq:alpha2u}\\
&\log\sqrt{\alpha} < \min_{(i,j) \in E^c} ([J_{\bm{\varepsilon}} \bm{f}]_{ij} +\hat{x}_j)- \max_i \hat{x}_i,	\label{eq:alpha3u}\\
&\log(\alpha+\sqrt{\alpha}) < \min_i \hat{x}_i-\max_i \hat{x}_i.	\label{eq:alpha4u}
\end{align}

Since $\bm{f}$ is tropical-linearly approximated at $\bm{\varepsilon}$, there exists $\delta_0 > 0$ such that
\begin{align}
	\|\bm{x}\| < \delta_0 \quad \Rightarrow \quad 
	\max_{i=1,2,\dots,n} d(f_i(\bm{x}), [J_{\bm{\varepsilon}}\bm{f} \otimes \bm{x}]_{i})
	 \leq \alpha \|\bm{x}\|.
	\label{eq:linapprox}
\end{align}
Contrary to the assertion of the theorem, let us assume that $\bm{\varepsilon}$ is a stable fixed point. Then, there exists $\delta > 0$ such that 
\begin{align}
	\|\bm{x}^{(0)}\| < \delta \quad \Rightarrow \quad \|\bm{x}^{(t)}\| < \delta_0	 \text{ for all } t \in \mathbb{Z}_{\geq 0}.	\label{eq:oncont}
\end{align}

By choosing a sufficiently small $\rho \in \mathbb{R}$,
we take an initial vector $\bm{x}^{(0)} := \rho \otimes \hat{\bm{x}}$ such that $\|\bm{x}^{(0)}\| < \delta$.
Since $\lambda-\eta/2 \geq \lambda/2 > 0$, the inequality
\begin{align}
	\max_{i=1,2,\dots,n} (x^{(t)}_i -\hat{x}_i) \geq \left(\lambda-\frac{\eta}{2}\right)t + \rho	,\qquad t \in \mathbb{Z}_{\geq 0}, \label{eq:unstable}
\end{align}
will lead to $\lim_{t \to \infty} \max_i x^{(t)}_i = \infty$, which contradicts the stability.
Now, we show~\eqref{eq:unstable} together with the fact that the maximum value on the left-hand side is attained by some $i \in V^c$ by induction on $t$.

The case where $t=0$ is trivial because $x_i^{(0)} - \hat{x}_i = \rho$ for all $i=1,2,\dots,n$.
Let us assume that the claim is true for some $t \in \mathbb{Z}_{\geq 0}$.
Let $i_1 \in V^c$ be an index $i$ that maximizes $x^{(t)}_{i} -\hat{x}_{i}$. Then, $(i_2,i_1) \in E^c$ for some $i_2 \in V^c$. 
Note that 
\begin{align}
	\max_i x^{(t)}_i - \max_i \hat{x}_i \leq \max_i (x^{(t)}_i - \hat{x}_i) = x^{(t)}_{i_1} - \hat{x}_{i_1}.		\label{eq:max}
\end{align}
Combining the above inequality with \eqref{eq:alpha3u}, we obtain
\begin{align*}
	[J_{\bm{\varepsilon}} \bm{f} \otimes \bm{x}^{(t)}]_{i_2} &\geq [J_{\bm{\varepsilon}} \bm{f}]_{i_2i_1} \otimes x^{(t)}_{i_1} \\
	&\geq ([J_{\bm{\varepsilon}} \bm{f}]_{i_2i_1} +\hat{x}_{i_1}- \max_i \hat{x}_i) + \max_i x^{(t)}_i \\
	&> \log \sqrt{\alpha} + \max_i x^{(t)}_i.
\end{align*}
Using \eqref{eq:linapprox} and \eqref{eq:oncont} and noting that $\bm{f}(\bm{x}^{(t)}) = \bm{x}^{(t+1)}$, we have 
$d(x_{i}^{(t+1)}, [J_{\bm{\varepsilon}}\bm{f} \otimes \bm{x}^{(t)}]_{i})  \leq \alpha \|\bm{x}^{(t)}\|$ for all $i=1,2,\dots,n$, yielding
\begin{align*}
	-\alpha e^{\max_i x^{(t)}_i} \leq e^{x_{i_2}^{(t+1)}} - e^{[J_{\bm{\varepsilon}} \bm{f} \otimes \bm{x}^{(t)}]_{i_2}} \leq \alpha e^{\max_i x^{(t)}_i}.
\end{align*}
This leads to
\begin{align*}
	e^{x_{i_2}^{(t+1)}} &\geq e^{[J_{\bm{\varepsilon}} \bm{f} \otimes \bm{x}^{(t)}]_{i_2}} - \alpha e^{\max_i x^{(t)}_i} \\
	&> e^{[J_{\bm{\varepsilon}} \bm{f} \otimes \bm{x}^{(t)}]_{i_2}} - \alpha e^{(-\log\sqrt{\alpha}) + [J_{\bm{\varepsilon}} \bm{f} \otimes \bm{x}^{(t)}]_{i_2}} \\
	&= e^{[J_{\bm{\varepsilon}} \bm{f} \otimes \bm{x}^{(t)}]_{i_2}} - \sqrt{\alpha}\, e^{[J_{\bm{\varepsilon}} \bm{f} \otimes \bm{x}^{(t)}]_{i_2}} \\
	& = (1-\sqrt{\alpha}) e^{[J_{\bm{\varepsilon}} \bm{f} \otimes \bm{x}^{(t)}]_{i_2}}  \\
	&= e^{\log(1-\sqrt{\alpha}) + [J_{\bm{\varepsilon}} \bm{f} \otimes \bm{x}^{(t)}]_{i_2}}.
\end{align*}
Hence, from \eqref{eq:visualize} and \eqref{eq:alpha2u}, we obtain
\begin{align*}
	x^{(t+1)}_{i_2} &> \log(1-\sqrt{\alpha}) + [J_{\bm{\varepsilon}} \bm{f} \otimes \bm{x}^{(t)}]_{i_2} \\
	&> \left(-\frac{\eta}{2}\right) + [J_{\bm{\varepsilon}} \bm{f}]_{i_2i_1} + x^{(t)}_{i_1} \\
	& = \left(-\frac{\eta}{2}\right) + ([J_{\bm{\varepsilon}} \bm{f}]_{i_2i_1} \otimes \hat{x}_{i_1}) - \hat{x}_{i_1} + x^{(t)}_{i_1} \\
	&= \left(-\frac{\eta}{2}\right) + (\lambda \otimes \hat{x}_{i_2}) + (x^{(t)}_{i_1} - \hat{x}_{i_1}),
\end{align*}
which implies
\begin{align}
	x^{(t+1)}_{i_2} - \hat{x}_{i_2} &> \left(\lambda-\frac{\eta}{2}\right) + ( x^{(t)}_{i_1} - \hat{x}_{i_1}). 		\label{eq:tinc} 
\end{align}
By induction, we obtain
\begin{align*}
	x^{(t+1)}_{i_2} - \hat{x}_{i_2}\geq \left(\lambda-\frac{\eta}{2}\right)(t+1) + \rho, 
\end{align*}
which proves \eqref{eq:unstable} for $t+1$.

We next prove that the maximum of $x^{(t+1)}_i -\hat{x}_i$ is attained by some $i \in V^c$. 
For any $i \not\in V^c$ and $j = 1,2,\dots,n$, we have
\begin{align*}
	([J_{\bm{\varepsilon}} \bm{f}]_{ij} \otimes x^{(t)}_j) - \hat{x}_i
	&=  ([J_{\bm{\varepsilon}} \bm{f}]_{ij} \otimes \hat{x}_j)  - \hat{x}_j + x^{(t)}_j - \hat{x}_i \\
	&\leq ((\lambda \otimes \hat{x}_i) - \eta) - \hat{x}_j + x^{(t)}_j - \hat{x}_i \\
	&= (\lambda -\eta) + (x^{(t)}_j - \hat{x}_j).
\end{align*}
because of the fact that $(i,j) \not\in E^c$ and the definition of $\eta$ in~\eqref{eq:defeta}.
By taking the maximum for $j=1,2,\dots,n$, we have
\begin{align}
	 [J_{\bm{\varepsilon}} \bm{f} \otimes \bm{x}^{(t)}]_i - \hat{x}_i \leq (\lambda -\eta) + \max_j (x^{(t)}_j - \hat{x}_j).	\label{eq:noncri}
\end{align}
Recall that $\|\bm{x}^{(t)}\| < \delta_0$ induces $d(f_i(\bm{x}^{(t)}), [J_{\bm{\varepsilon}}\bm{f} \otimes \bm{x}^{(t)}]_{i})  \leq \alpha \|\bm{x}^{(t)}\|$.
Combining Lemma~\ref{lem:ub} with \eqref{eq:alpha1u} and \eqref{eq:alpha4u}, we obtain
\begin{align*}
	x^{(t+1)}_i 
	= f_i(\bm{x}^{(t)})
	\leq \frac{\eta}{2} \otimes [J_{\bm{\varepsilon}} \bm{f} \otimes \bm{x}^{(t)}]_i  
	\oplus (\min_j \hat{x}_j - \max_j \hat{x}_j) \otimes (\bm{0} \otimes \bm{x}^{(t)}).
\end{align*}
Then, using \eqref{eq:max}, \eqref{eq:tinc} and \eqref{eq:noncri}, we obtain
\begin{align*}
	x^{(t+1)}_i - \hat{x}_i 
	&\leq \left(\frac{\eta}{2} + [J_{\bm{\varepsilon}} \bm{f} \otimes \bm{x}^{(t)}]_i - \hat{x}_i \right) 
	\oplus (\min_j \hat{x}_j - \max_j \hat{x}_j + \max_j x^{(t)}_j - \hat{x}_i )\\
	& \leq \left(\frac{\eta}{2} + (\lambda -\eta) + \max_j (x^{(t)}_j - \hat{x}_j) \right)
	\oplus (x^{(t)}_{i_1} - \hat{x}_{i_1} + \min_j \hat{x}_j  - \hat{x}_i) \\
	& \leq \left(\lambda -\frac{\eta}{2} + (x^{(t)}_{i_1} - \hat{x}_{i_1}) \right) \oplus  (x^{(t)}_{i_1} - \hat{x}_{i_1}) \\
	&= \left(\lambda -\frac{\eta}{2} \right) + (x^{(t)}_{i_1} - \hat{x}_{i_1}) \\
	&< x^{(t+1)}_{i_2} - \hat{x}_{i_2}. 
\end{align*}
Considering that the above inequality holds for all $i \not\in V^c$, we proved that $\max_i (x^{(t+1)}_i - \hat{x}_i)$ is attained by some $i \in V^c$.
Thus, the claim is also true for $t+1$.
By induction, we have proved \eqref{eq:unstable} for all $t \in \mathbb{Z}_{\geq 0}$.
This completes the proof of the theorem.
\endproof

We demonstrate Theorems~\ref{thm:main1} and~\ref{thm:main2} in some examples.

\begin{example}
Let us consider a difference equation
\begin{align*}
	\begin{pmatrix} x^{(t+1)} \\ y^{(t+1)} \end{pmatrix} 
	= \begin{pmatrix} (x^{(t)} \oplus a) \otimes y^{(t)} \\ 1 \otimes x^{(t)} \oplus (-1) \otimes y^{(t)} \end{pmatrix}.
\end{align*}
As each entry of $\bm{f}(x,y):= ((x \oplus a) \otimes y, 1 \otimes x \oplus (-1) \otimes y)^\top$ is a tropical polynomial function, it is tropical-linearly approximated as
\begin{align*}
	\bm{f}(x,y) \approx \begin{pmatrix} \varepsilon & a \\ 1 & -1 \end{pmatrix} \otimes \begin{pmatrix} x \\ y \end{pmatrix}.
\end{align*}
Since the maximum eigenvalue of $\begin{pmatrix} \varepsilon & a \\ 1 & -1 \end{pmatrix}$ is $\max( (a+1)/2, -1)$, the fixed point $\bm{\varepsilon}$ is asymptotically stable if $a < -1$ and unstable if $a > -1$.
\end{example}

\begin{example}
Let us define
\begin{align*}
	H_1(X,Y) &=   X \otimes (X \oplus T \otimes F_1(X,Y)) \oslash (X \oplus T \otimes G_1(X,Y)), \\
	H_2(X,Y) &= Y \otimes (Y \oplus T \otimes F_2(X,Y)) \oslash (Y \oplus T \otimes G_2(X,Y)),
\end{align*}
where $T$ is a constant, and consider a difference equation
\begin{align}
	\begin{pmatrix}  X^{(t+1)} \\ Y^{(t+1)} \end{pmatrix}
	= \begin{pmatrix} H_1(X^{(t)},Y^{(t)})  \\	H_2(X^{(t)},Y^{(t)}) \end{pmatrix}.
	\label{eq:example2}
\end{align}
This kind of system is derived through the tropical discretization of differential equations
\begin{align*}
	\begin{cases} \displaystyle \frac{dx}{dt} =f_1(x,y)-g_1(x,y), \\[7pt] \displaystyle  \frac{dy}{dt} = f_2(x,y)-g_2(x,y). \end{cases}
\end{align*}
See \cite{Murata2013} for details on tropical discretization. 

We assume that $F_i(X,Y)$ and $G_i(X,Y)$ converge to $F_i(\varepsilon,\varepsilon)$ and $G_i(\varepsilon,\varepsilon)$, respectively, for $i=1,2$ as $(X,Y) \to (\varepsilon,\varepsilon)$.
If $\bm{\varepsilon}$ is a fixed point of~\eqref{eq:example2}, then $G_i(\varepsilon,\varepsilon) \neq \varepsilon$ for $i=1,2$.
The tropical derivatives at $\bm{\varepsilon}$ are computed as
\begin{align*}
	D_{X,\bm{\varepsilon}}H_1(X,Y) &= F_1(\varepsilon,\varepsilon) \oslash G_1(\varepsilon,\varepsilon), \\
	D_{Y,\bm{\varepsilon}}H_2(X,Y) &= F_2(\varepsilon,\varepsilon) \oslash G_2(\varepsilon,\varepsilon), \\
	D_{Y,\bm{\varepsilon}}H_1(X,Y) &= D_{X,\bm{\varepsilon}}H_2(X,Y) = \varepsilon,
\end{align*}
and $H_1(X,Y)$ and $H_2(X,Y)$ are tropical-linearly approximated as
\begin{align*}
	\begin{pmatrix} H_1(X,Y) \\ H_2(X,Y) \end{pmatrix}
	\approx \begin{pmatrix} F_1(\varepsilon,\varepsilon) \oslash G_1(\varepsilon,\varepsilon) & \varepsilon \\
	\varepsilon & F_2(\varepsilon,\varepsilon) \oslash G_2(\varepsilon,\varepsilon) \end{pmatrix}
	\otimes \begin{pmatrix} X \\ Y \end{pmatrix}.
\end{align*}
The fixed point $\bm{\varepsilon}$ is asymptotically stable if both $F_1(\varepsilon,\varepsilon) \oslash G_1(\varepsilon,\varepsilon)$ and $F_2(\varepsilon,\varepsilon) \oslash G_2(\varepsilon,\varepsilon)$ are negative, and it is unstable if either of them is positive.
\end{example}

\begin{example}
Let us consider a discrete event system on $n$ processors $P_1, P_2, \dots, P_n$ that process different products.
At each time step, a processor receives products from other processors and then processes their own products as many as possible, using one unit from each per unit of output.
Let $y_i^{(t)}$ be the cumulative number of products processed in $P_i$ during time steps $0,1,\dots,t$.
If the processor $P_i$ requires products from $P_{j_1}, P_{j_2}, \dots, P_{j_r}$, the number of products processed in $P_i$ up to times step $t+1$ is 
\begin{align*}
	y_i^{(t+1)} = \min_{k=1,2,\dots,r} (y_{j_k}^{(t)} - a_{i,j_k}(\bm{y}^{(t)})).
\end{align*}
Here, $a_{i,j_k}(\bm{y})$ represents changes in the product quantity due to other factors; when $a_{i,j_k}(\bm{y})$ is positive, it corresponds to losses occurring during the process, whereas when $a_{i,j_k}(\bm{y})$ is negative, it represents an external supply.
By setting $y_i^{(t)} = -x_i^{(t)}$ for $i=1,2,\dots,n$, this min-plus system can be switched to a max-plus (tropical) model as
\begin{align*}
	\bm{x}^{(t)} = A(\bm{x}^{(t)}) \otimes \bm{x}^{(t)}.
\end{align*}
If the limit $A:= \lim_{\bm{x} \to \bm{\varepsilon}} A(\bm{x})$ exists and the maximum eigenvalue of $A$ is negative, then the fixed point $\bm{\varepsilon}$ is asymptotically stable. In terms of the min-plus model for $\bm{y}^{(t)}$, the system operates continuously without encountering a deadlock when it starts from sufficiently large $\bm{y}^{(0)}$.
\end{example}

\section{Concluding remarks}

In this study, we proposed a tropical linear approximation approach for the stability analysis of difference equations at the tropical origin, namely, $\bm{\varepsilon}$. A natural question is how to expand this approach to any fixed points in $\mathbb{R}_{\max}^n$, especially in $\mathbb{R}^n$.
This case is very different from the case where $\bm{\varepsilon}$ is a fixed point.

Let us assume that $\bm{p} = (p_1,p_2,\dots,p_n)^\top \in \mathbb{R}^n$ is a fixed point of~\eqref{eq:dyn}. If $A \otimes \bm{x} \oplus \bm{b}$ is a tropical linear approximation of $\bm{f}$ at $\bm{p}$, then $\bm{p} = A \otimes \bm{p} \oplus \bm{b}$ should be satisfied.
This linear equation for $A = (a_{i,j})$ and $\bm{b} = (b_1,b_2,\dots,b_n)^\top$ can be solved as
\begin{align}
\begin{cases} a_{1,1} \leq 0 \\ a_{1,2} \leq p_1-p_2 \\ \qquad\vdots \\ a_{1,n} \leq p_1-p_n \\ b_1 \leq p_1 \end{cases},\ 
\begin{cases} a_{2,1} \leq p_2-p_1  \\ a_{2,2} \leq 0\\ \qquad\vdots \\ a_{2,n} \leq p_2-p_n \\ b_2 \leq p_2 \end{cases},\ 
\dots,\ 
\begin{cases} a_{n,1} \leq p_n-p_1  \\ a_{n,2} \leq p_n-p_2 \\ \qquad\vdots \\ a_{n,n} \leq 0\\ b_n \leq p_n \end{cases},
\label{eq:conclusion}
\end{align}
where at least one equality holds for each collection of inequalities.
Moreover, the term $a_{i,j} \otimes x_j$ contributes to evaluating $A \otimes \bm{x} \oplus \bm{b}$ around $\bm{p}$ only if $a_{i,j} = p_i-p_j$.
Hence, the tropical linear approximation is determined by the fixed point $\bm{p}$ itself rather than the infinitesimal behaviour around $\bm{p}$.

Furthermore, if a tropical linear difference equation $\bm{x}^{(t+1)} = A \otimes \bm{x}^{(t)} \oplus \bm{b}$ has a fixed point $\bm{p}$, then we have 
\begin{align*}
	\bm{p} = A \otimes \bm{p} \oplus \bigoplus_{k=0}^m A^{\otimes k} \otimes \bm{b}
\end{align*}
for any $m \in \mathbb{Z}_{\geq 0}$ by using $\bm{p} = A \otimes \bm{p} \oplus \bm{b}$ repeatedly.
When $\bm{p}$ and $\bm{b}$ are finite vectors, the matrix sequence $\{A^{\otimes k}\}_{k \in \mathbb{Z}_{\geq 0}}$ should be bounded from above.
This implies that the maximum eigenvalue of $A$ must be nonpositive. 
Therefore, the stability analysis based on the sign of the maximum eigenvalue would exhibit substantially different behavior.
Nonetheless, it would be interesting to investigate how the behavior differs when the maximum eigenvalue is negative versus when it is zero.
This is related to whether each inequality $a_{i,j} \leq p_i-p_j$ in~\eqref{eq:conclusion} is strict or not.
Hence, the stability at the finite fixed point might be explained in terms of the tropical linearization matrix $A$ as well, which is left for future study.

\section*{Acknowledgement(s)}

This work was supported by JSPS KAKENHI (Grant No.~22K13964).

\section*{Disclosure statement}

The authors report there are no competing interests to declare.

%\bibliographystyle{abbrv}
%\bibliography{reference}

\end{document}